\documentclass[12pt,reqno]{amsart}
\usepackage{amsthm,amsfonts,amssymb,euscript}

\newcommand{\bea}{\begin{eqnarray}}
\newcommand{\eea}{\end{eqnarray}}
\def\beaa{\begin{eqnarray*}}
\def\eeaa{\end{eqnarray*}}
\def\ba{\begin{array}}
\def\ea{\end{array}}
\def\be#1{\begin{equation} \label{#1}}
\def \eeq{\end{equation}}

\def\a{{\alpha}}

\def\be{{\beta}}
\def\ga{\gamma}

\def\ep{\epsilon}

\def\al{\alpha}

\def\c{\cdot}

\def\H{{\mathbb{H}}}
\def\R{{\mathbb{R}}}
\def\C{{\mathbb{C}}}
\def\S{{\mathbb{S}}}

\def\D{{\bf D}}

\def\c{{\bf c}}
\def\g{{\bf g}}

\DeclareMathOperator{\ch}{ch}
\DeclareMathOperator{\sh}{sh}

\newtheorem{theorem}{Theorem}[section]
\newtheorem{lemma}[theorem]{Lemma}
\newtheorem{proposition}[theorem]{Proposition}

\newtheorem{remark}[theorem]{Remark}

\numberwithin{equation}{section}

\begin{document}

\title[Global well-posedness and scattering in $H^1$]{Semilinear Schr\"{o}dinger flows on hyperbolic spaces: scattering in $H^1$}

\author{Alexandru D. Ionescu}
\address{University of Wisconsin--Madison}
\email{ionescu@math.wisc.edu}

\author{Gigliola Staffilani}
\address{Massachusetts Institute of Technology}
\email{gigliola@math.mit.edu}

\thanks{The first author was supported in part by a Packard Fellowship. The second author was supported in part by NSF Grant DMS 0602678.}

\begin{abstract}
We prove global well-posedness and scattering in $H^1$ for the defocusing nonlinear Schr\"{o}dinger equations
\begin{equation*}
\begin{cases}
&(i\partial_t+\Delta_\g)u=u|u|^{2\sigma};\\
&u(0)=\phi,
\end{cases}
\end{equation*}
on the hyperbolic spaces $\H^d$, $d\geq 2$, for exponents $\sigma\in(0,2/(d-2))$. The main unexpected conclusion is scattering to linear solutions in the case of small exponents $\sigma$; for comparison, on Euclidean spaces scattering in $H^1$ is not known for any exponent $\sigma\in(1/d,2/d]$ and is known to fail for $\sigma\in(0,1/d]$. Our main ingredients are certain noneuclidean global in time Strichartz estimates and noneuclidean Morawetz inequalities.  
\end{abstract}
\maketitle
\tableofcontents

\section{Introduction}\label{Intro}

In this paper we consider semilinear Schr\"{o}dinger initial-value problems of the form
\begin{equation}\label{eq1}
\begin{cases}
&(i\partial_t+\Delta_\g)u=N(u);\\
&u(0)=\phi,
\end{cases}
\end{equation}
on Riemannian manifolds $(M,\g)$ of dimensions $d\geq 2$. Typical nonlinearity are
\begin{equation}\label{NU}
N(u)=\lambda u|u|^{2\sigma}\text{ for suitable }\sigma\in(0,\infty),\,\,\lambda=\pm1.
\end{equation}
The initial-value problem \eqref{eq1} has been studied extensively in the Euclidean geometry for large classes of nonlinearities, see the recent books \cite{Cazenave:book} and \cite{Tao:book}, and the references therein. For example, on Euclidean spaces, it is known that the defocusing $H^1$ subcritical initial-value problem \eqref{eq1}, where $N(u)=u|u|^{2\sigma}$, $\sigma\in(0,2/(d-2))$, is globally well-posed in the energy space $H^1$; moreover, scattering to linear solutions is known in the restricted range $\sigma\in(2/d,2/(d-2))$, see \cite{GiVe} and \cite{Na}. In recent years, the more delicate problems that correspond to critical power nonlinearities, both in $\dot{H}^1$ and $L^2$, have also been considered, see \cite{B, G, Tcrit, CKSTTcrit, RV, V, KeMe, TVZ, KTV}. 

The initial-value problem \eqref{eq1} has also been considered in the setting of compact Riemannian manifolds $(M,\g)$, see \cite{Bo2, Bo1, BuGeTz, BuGeTz2}. In this case the conclusions are generally weaker than in Euclidean spaces: there is no scattering to linear solutions, or some other type of asymptotic control of the nonlinear evolution as $t\to\infty$. Moreover, in certain cases such as the spheres $\mathbb{S}^d$, the well-posedness theory requires sufficiently subcritical nonlinearities, due to concentration of certain spherical harmonics, see \cite{BuGeTz3}.

In this paper we consider the initial-value problem \eqref{eq1} in the setting of symmetric spaces of noncompact type\footnote{The symmetric spaces of noncompact type are simply connected Riemannian manifolds of nonpositive sectional curvature, without Euclidean factors, and for which every geodesic symmetry defines an isometry.}. The simplest such spaces are the hyperbolic spaces $\H^d$, $d\geq 2$. On hyperbolic spaces we prove {\it{stronger}} theorems than on Euclidean spaces. For the linear flow we prove a larger class of global in time Strichartz estimates (see \eqref{Strex} or Proposition \ref{Strichartz}, for radial functions these were already proved in \cite{Ba, BaDu, Pi, BaCaSt}). For the nonlinear flow with $N(u)=u|u|^{2\sigma}$, $\sigma\in(0,2/(d-2))$, we prove noneuclidean Morawetz inequalities, such as \eqref{Moraex}, which lead to large data scattering in $H^1$ in the full subcritical range $\sigma\in(0,2/(d-2))$. These stronger theorems are possible because of the more robust geometry at infinity of noncompact symmetric spaces compared to Euclidean spaces; for example, the scattering result for the nonlinear Schr\"{o}dinger equation can be interpreted as the absence of long range effects of the nonlinearity.  This absence of long range effects was already observed in  \cite{BaCaSt} under radial symmetry assumptions for $d=3$. 

After we finished typing this paper we learned that J.-P. Anker and V. Pierfelice also had obtained  Strichartz estimates  similar to the ones we prove here and  small data scattering results (see \cite{AnPi}).

Assume from now on that $M=\H^d$ is a hyperbolic space and $N(u)=u|u|^{2\sigma}$, $\sigma\in(0,2/(d-2))$. Suitable solutions on the time interval $(-T,T)$ of \eqref{eq1} satisfy (at least formally) mass and energy conservation,
\begin{equation}\label{conserve}
\begin{split}
&E^0(u)(t):=||u(t)||_{L^2(\H^d)}=E^0(u)(0);\\
&E^1(u)(t):=\frac{1}{2}\int_{\H^d}|\nabla u(t)|^2\,d\mu+\frac{1}{2\sigma+2}\int_{\H^d}|u(t)|^{2\sigma+2}\,d\mu=E^1(u)(0),
\end{split}
\end{equation}
for any $t\in(-T,T)$. Our main theorem concerns global well-posedness and scattering in $H^1$ for the initial-value problem \eqref{eq1}. To state the theorem precisely we need to define the Strichartz space $S^1_q(I)$. For any bounded, open interval $I=(a,b)\subseteq\R$ and any exponent $q\in(2,(2d+4)/d]$, let $r=2dq/(dq-4)\in[(2d+4)/d,2d/(d-2))$ and define the Banach space 
\begin{equation}\label{keyspace1}
\begin{split}
S^0_q(I)=\{f\in C(I:L^2(\H^d)):||f||_{S^0_q(I)}=\sup\big[||f||_{L^{\infty,2}_I},||f||_{L^{q,r}_I},||f||_{L^{q,q}_I}\big]<\infty\},
\end{split}
\end{equation}
where, by definition,
\begin{equation*}
||f||_{L^{p_1,p_2}_I}=\Big[\int_{I}\Big(\int_{\H^d}|f(t,x)|^{p_2}\,dx\Big)^{p_1/p_2}\,dt\Big]^{1/p_1}
\end{equation*} 
is the usual space-time Strichartz norm on the time interval $I$. Notice that the pair $(q,r)$ in \eqref{keyspace1} is an admissible pair, i.e. 
\begin{equation*}
\frac{2}{q}=d\Big(\frac{1}{2}-\frac{1}{r}\Big)\quad \text{ and }\quad(q,r)\in[2,\infty]\times[2,\infty),
\end{equation*}
and, in addition, $2<q\leq r$. We define the Banach space $S^1_q(I)$\footnote{On hyperbolic spaces, the symbol of the operator $-\Delta$ is the multiplier $(\lambda^2+\rho^2)$, see section \ref{preliminaries} for notation; in particular $||f||_{L^p(\H^d)}\leq C_p||(-\Delta)^{1/2}f||_{L^p(\H^d)}$ for any $p\in(1,\infty)$ therefore $S^1_q(I)\hookrightarrow S^0_q(I)$.},
\begin{equation}\label{keyspace2}
\begin{split}
 S^1_q(I)=\{f\in C(I:H^1(\H^d)):||f||_{S^1_q(I)}=||(-\Delta)^{1/2}(f)||_{S^0_q(I)}<\infty\}.
\end{split}
\end{equation}
We state now our main theorem.

\begin{theorem}\label{Main1}
Assume $\sigma\in (0,2/(d-2))$ and $q\in(2,(2d+4)/d]$ is fixed.

(a) (Global well-posedness) If $\phi\in H^1(\H^d)$ then there exists a unique global solution $u\in C(\R:H^1(\H^d))$ of the initial-value problem
\begin{equation}\label{eq1.1}
\begin{cases}
&(i\partial_t+\Delta_\g)u=u|u|^{2\sigma}\text{ in }C(\R:H^{-1}(\H^d));\\
&u(0)=\phi.
\end{cases}
\end{equation}
In addition, for any $T\in[0,\infty)$, the mapping $$\phi\to U_T(\phi)=\mathbf{1}_{(-T,T)}(t)\cdot u$$ is a continuous mapping from $H^1(\H^d)$ to $ S^1_q(-T,T)$, and the conservation laws \eqref{conserve} are satisfied.

(b) (Scattering) We have the uniform bound
\begin{equation}\label{sol2}
||u||_{S^1_q(-T,T)}\leq C(\sigma,q,||\phi||_{H^1(\H^d)})
\end{equation}
for any $T\in[0,\infty)$. As a consequence, there exist unique $u_{\pm}\in H^1(\H^d)$ such that
\begin{equation}\label{sol4}
||u(t)-W(t)u_{\pm}||_{H^1(\H^d)}=0\text{ as }t\to\pm\infty.
\end{equation}   
\end{theorem}

The main conclusion of the theorem is the $H^1$ scattering and the uniform bound \eqref{sol2}, particularly in the case of small exponents $\sigma\in(0,2/d]$. We also emphasize that this theorem does not require radial symmetry. We recall that the existence of the wave operator was already proved in \cite{BaCaSt} (for $\sigma\in(0,2/(d-2))$), under a radial symmetry assumption. In this same paper scattering was only proved  for $d=3$, still in the range  $\sigma\in(0,2)$ and under the radial symmetry assumption. Without the radial symmetry condition, scattering was only proved there  for $\sigma\in(2/3,2)$, $d=3$ . 

Simultaneously to this work in \cite{BaCaDu} the authors  also proved  $H^1$ scattering for 
$\sigma\in(0,2/(d-2))$ and $d\geq 3$ under radial symmetry assumptions.

We make several remarks.

\begin{remark}
It is easy to see that the proof of Theorem \ref{Main1} extends to more general defocusing $H^1$ subcritical nonlinearities. Theorem \ref{Main1} also extends to the more general setting of noncompact symmetric spaces of real rank one, since the key ingredients, such as the inequality \eqref{special}, the formulas \eqref{Ke4} and \eqref{Ke6}, and Lemma \ref{functiona}, have suitable analogues on noncompact symmetric spaces of real rank one. We stated Theorem \ref{Main1} in the special case of power nonlinearities and in the setting of hyperbolic spaces mostly for the sake of concreteness. On the other hand, on general noncompact symmetric spaces, simple identities such as \eqref{Ke4} and \eqref{Ke6} do not hold (with uniform bounds on the symbols), and the analysis on such spaces is more delicate, see for example \cite{AnJi} for bounds on the heat kernel. We hope to return to such questions, as well as the more delicate questions related to critical nonlinearities.  
\end{remark}

\begin{remark}
The  conclusion of Theorem \ref{Main1} (b) is in sharp contrast with its Euclidean analogue: on Euclidean spaces scattering to linear solutions is only known for exponents $\sigma\in(2/d,2/(d-2))$, see \cite{GiVe, Na, CKSTTscat, CHVZ, TVZmixed}, and also for $\sigma=2/(d-2)$, see \cite{B, CKSTTcrit, Tcrit, RV, V}. Moreover, on Euclidean spaces, scattering in $H^1$ is known to fail for exponents $\sigma\in(0,1/d]$, see \cite{Barab, Strauss}. Even  the easier question of existence of the wave operator  in the range $(1/d,2/d]$  is not settled yet, and in particular  no (unweighted) scattering results are known for  the  range $\sigma\in(1/d,2/d].$ Recently  scattering in $L^2$ has been proved for $\sigma=2/d$ under the radial symmetry assumption \cite{KTV, TVZ}.
\end{remark}

\begin{remark}
We observe that in the case in which $\sigma$ is a natural number one can also prove preservation of regularity in the sense that for any $\phi\in H^s$, $s\geq 1$, the solution $u$ to \eqref{eq1.1} is in $C(\R:H^s(\H^d))$. Combining this fact with the scattering result above one can easily show that
\begin{equation}\label{wt}
\|u(t)\|_{H^s}\lesssim C_s \, \, \, \mbox{ as } \, \, t \rightarrow \pm\infty.
\end{equation}
This global estimate can be also interpreted as lack of weak turbulence for the initial value problem \eqref{eq1.1}. In the Euclidean spaces a similar result is  also available, but only for  $\sigma\in(2/d,2/(d-2))$, in particular a bound as in \eqref{wt} is still unknown for the smooth solutions of the cubic NLS when  $d=2$, $\sigma=1$, see \cite{CDKS}.
\end{remark}

We describe now the two main noneuclidean ingredients needed in the proof of Theorem \ref{Main1} (b). The first ingredient is the inequality \eqref{special}
\begin{equation}\label{special2}
\|f\ast |K|\,\|_{L^2(\H^d)}\leq C\|f\|_{L^2(\H^d)}\cdot\int_{0}^\infty |K(r)|e^{-\rho r}(r+1)(\sh r)^{2\rho}\,dr, 
\end{equation}
for any $f,K\in C_0^\infty(\H^d)$, provided that $K$ is a radial kernel. For comparison, $\|K\|_{L^1(\H^d)}=C\int_{0}^\infty |K(r)|(\sh r)^{2\rho}\,dr$, thus the factor $e^{-\rho r}(r+1)$ in \eqref{special2} represents a nontrivial gain\footnote{This gain is related to the Kunze--Stein phenomenon $L^2(\mathbb{G})\ast L^p(\mathbb{G})\subseteq L^2(\mathbb{G})$ for any $p\in[1,2)$, where $\mathbb{G}$ is the Lorentz group $SO(d,1)$, see \cite{KuSt} and \cite{Co}.} over the (best possible) Euclidean inequality $\|f\ast |K|\,\|_{L^2(\R^d)}\leq \|f\|_{L^2(\R^d)}\|K\|_{L^1(\R^d)}$. We exploit this gain in section \ref{Strichartz} to prove the noneuclidean Strichartz estimates (as well as suitable inhomogeneous estimates)
\begin{equation}\label{Strex}
\|W(t)\phi\|_{L^{q}(\R\times\H^d)}\leq C_q\|\phi\|_{L^2}\text{ for any }q\in(2,(2d+4)/d]. 
\end{equation}
In Section \ref{Strichartz} we prove these Strichartz estimates, together with the nonendpoint Euclidean-type Strichartz estimates\footnote{These were already proved locally in time in \cite{Ba}. However, since our main goal is to prove scattering, we have to obtain global in time estimates.} (and suitable inhomogeneous estimates)$$\|W(t)\phi\|_{L^{\infty,2}_I\cap L^{q,r}_I}\leq C_q\|\phi\|_{L^2}.$$

The second noneuclidean ingredient we need is the existence of a smooth radial function $a:\H^d\to[0,\infty)$ with the properties
\begin{equation*}
\Delta a=1\text{ and }|\nabla a|\leq C\quad\text{ on }\H^d.
\end{equation*}
We construct such a function, with the additional property $\D^2 a\geq 0$, in Lemma \ref{functiona}. In section \ref{Mora}, we use this function and standard arguments to prove the Morawetz inequality
\begin{equation}\label{Moraex}
\|u\|^{2\sigma+2}_{L^{2\sigma+2}((-T,T)\times\H^d)}\leq C_\sigma\sup_{t\in (-T,T)}\|u(t)\|_{L^2(\H^d)}\|u(t)\|_{H^1(\H^d)},
\end{equation}
for any solution $u\in S^1_q(-T,T)$ of the nonlinear Schr\"{o}dinger equation, with a constant $C$ that does not depend on $T$. Theorem \ref{Main1} (b) follows, using mostly standard arguments, from this Morawetz inequality and the noneuclidean Strichartz estimates described above (in the most interesting case $\sigma\in(0,2/d]$ we combine the Morawetz inequality \eqref{Moraex} with the noneuclidean Strichartz estimate \eqref{Strex} for $q=2\sigma+2$).

We discuss now the definition of the Strichartz spaces $S^j_q(I)$, $j=0,1$, see \eqref{keyspace1} and \eqref{keyspace2}. On bounded intervals $I$ (i.e. if $|I|\leq C$) the space $S^0_q(I)$ is equivalent to Euclidean-type Strichartz space $L^{\infty,2}_I\cap L^{q,r}_I$, where $(q,r)$ is an admissible pair with $q$ close to $2$\footnote{In Euclidean spaces, in dimensions $d\geq 3$, one can also use the endpoint pair $(q,r)=(2,2d/(d-2))$. Since this paper is concerned with theorems in all dimensions $d\geq 2$, we do not discuss the Keel-Tao \cite{KeTa} endpoint Strichartz estimate which holds in dimensions $d\geq 3$.}. This is because $$\|f\|_{L^{q,q}_I}\leq |I|^{1/q-1/p}\|f\|_{L^{p,q}_I}\leq C|I|^{1/q-1/p}(\|f\|_{L^{\infty,2}_I}+\|f\|_{L^{q,r}_I}),$$
where $p\in[(2d+4)/d,\infty)$ is such that $(p,q)$ is an admissible pair. The component $L^{q,q}_I$ in the definition of the space $S^0_q(I)$ is related to the noneuclidean Strichartz estimate \eqref{Strex} and the Morawetz inequality \eqref{Moraex}. This component becomes important in the proof of Theorem \ref{Main1} (b) since, to prove scattering, we need uniform estimates over long intervals. Using interpolation and Sobolev embedding, we have the $L^p$ bounds, uniformly in $|I|$,
\begin{equation}\label{Lpbound}
\|(-\Delta)^{1/2}f\|_{L^{p_1}(I\times\H^d)}+\|f\|_{L^{p_2}(I\times\H^d)}\leq C_{p_2}\|f\|_{S^1_q(I)},
\end{equation}    
for any $f\in S^1_q(I)$, $p_1\in[q,(2d+4)/d]$, and $p_2\in[q,(2d+4)/(d-2))$ (in dimensions $d\geq 3$ one can also take $p_2=(2d+4)/(d-2)$). In particular, $S^j_{q_1}(I)\hookrightarrow S^j_{q_2}(I)$ if $q_1\leq q_2$ and $j=0,1$, therefore the conclusions of Theorem \ref{Main1} become stronger as $q$ approaches $2$.

The rest of the paper is organized as follows: in section \ref{preliminaries} we introduce most of our notation related to hyperbolic spaces. The harmonic analysis and the geometry of hyperbolic spaces is very rich, due to the large groups of isometries they admit (the Lorentz groups $SO(d,1)$) and the special structure of these groups (semisimple Lie groups of real rank one). In section \ref{Strichartz} we prove our main Strichartz estimates, see Proposition \ref{Strichartz}. The global in time Strichartz estimates are new in the setting of hyperbolic spaces, with the exception of the case of radial functions (see  \cite{Ba}, \cite{Pi}, \cite{BaCaSt} for Strichartz estimates for radial functions). In section \ref{Mora} we prove our main Morawetz inequality, see Proposition \ref{MoraIneq}. Finally, in section \ref{proofScat} we use the standard combination of Strichartz estimates and Morawetz inequality to prove Theorem \ref{Main1}.   

\section{Preliminaries}\label{preliminaries}

In this section we discuss some aspects of the harmonic analysis and the geometry of hyperbolic spaces. The hyperbolic spaces are the simplest examples of noncompact symmetric spaces of real rank one (see for example the books \cite{He1} and \cite{He2} for a comprehensive view of the analysis and geometry on symmetric spaces and semisimple Lie groups). A short elementary exposition of some of the concepts needed in this paper, specialized to the case of hyperbolic spaces, can be found in \cite{Br}.  

\subsection{Riemannian structure and isometries of hyperbolic spaces} 
We consider the Minkowski space $\R^{d+1}$ with the standard Minkowski metric $-(dx^0)^2+(dx^1)^2+\ldots+(dx^d)^2$ and define the bilinear form on $\R^{d+1}\times\R^{d+1}$,
\begin{equation*}
[x,y]=x^0y^0-x^1y^1-\ldots-x^dy^d.
\end{equation*}
The hyperbolic space $\H^d$ is defined as 
\begin{equation*}
\H^d=\{x\in\R^{d+1}:[x,x]=1\text{ and }x^0>0\}.
\end{equation*}
Let ${\bf{0}}=(1,0,\ldots,0)$ denote the origin of $\H^d$. The Minkowski metric on $\R^{d+1}$ induces a Riemannian metric $\g$ on $\H^d$, with covariant derivative $\D$ and induced measure $d\mu$.

We define ${\mathbb G}=SO(d,1)=SO_e(d,1)$ as the connected Lie group of $(d+1)\times(d+1)$ matrices that leave the form $[.,.]$ invariant. Clearly, $X\in SO(d,1)$ if and only if
\begin{equation*}
{}^{tr}X\cdot I_{d,1}\cdot X=I_{d,1},\quad\det X=1,\quad X_{00}>0,
\end{equation*}
where $I_{d,1}$ is the diagonal matrix $\mathrm{diag}[-1,1,\ldots,1]$ (since $[x,y]=-{}^tx\cdot I_{d,1}\cdot y$). Let ${\mathbb K}=SO(d)$ denote the subgroup of $SO(d,1)$ that fix the origin ${\bf{0}}$. Clearly, $SO(d)$ is the compact rotation group acting on the variables $(x^1,\ldots,x^d)$. The hyperbolic space $\H^d$ can be identified with the homogeneous space $SO(d,1)/SO(d)={\mathbb G}/{\mathbb K}$, and the group $SO(d,1)$ acts transitively on $\H^d$. Moreover, for any $g\in SO(d,1)$ the mapping $L_g:\H^d\to\H^d$, $L_g(x)=g\cdot x$, defines an isometry of $\H^d$.

We introduce now a global system of coordinates on $\H^d$: we define the diffeomorphism $\Phi:\R^{d-1}\times\R\to\H^d$, 
\begin{equation}\label{ap1}
\Phi(v,s)={}^{tr}(\ch s+e^{-s}|v|^2/2,\sh s+e^{-s}|v|^2/2,e^{-s}v_1,\ldots,e^{-s}v_{d-1}),
\end{equation}
where $v=(v_1,\ldots,v_{d-1})\in\R^{d-1}$. This system of coordinates is related to the Iwasawa decomposition $\mathbb{G}=\mathbb{N}\mathbb{A}\mathbb{K}$, where $\mathbb{G}$ and $\mathbb{K}$ are the groups defined above, and $\mathbb{A}$, $\mathbb{N}$ are the subgroups of $\mathbb{G}$
\begin{equation}\label{subgroupa}
\mathbb{A}=\left\{a_s=\begin{bmatrix}
&\ch s&\sh s&0\\
&\sh s&\ch s&0\\
&0&0&I_{d-1}
\end{bmatrix}:t\in\R\right\},
\end{equation}
and
\begin{equation}\label{subgroupn}
\mathbb{N}=\left\{n_v=\begin{bmatrix}
&1+|v|^2/2&-|v|^2/2&{}^{tr}v\\
&|v|^2/2&1-|v|^2/2&{}^{tr}v\\
&v&-v&I_{d-1}
\end{bmatrix}:v\in\R^{d-1}\right\}.
\end{equation}
Thus $\Phi(v,s)=n_va_s\cdot {\bf{0}}$. Using this diffeomorphism we can identify $(\H^d,\g)$ with $\R^{d-1}\times\R$, with the induced Riemannian metric
\begin{equation*}
e^{-2s}[(dv_1)^2+\ldots+(dv_{d-1})^2+e^{2s}(ds)^2].
\end{equation*}
For $f\in C_0(\H^d)$ we have the integration formula
\begin{equation*}
\int_{\H^d}f(x)\,d\mu=\int_{\R^{d-1}\times\R}f(n_va_s\cdot{\bf{0}})e^{-2\rho s}\,dvds.
\end{equation*}
We fix the global orthonormal frame 
\begin{equation*}
e_\al=e^s\partial_{v_\al}\text{ for }\al=1,\ldots,d-1,\text{  and  }e_d=\partial_s.
\end{equation*}
We compute the commutators
\begin{equation*}
[e_d,e_\al]=e_\al,\quad [e_\al,e_\be]=[e_d,e_d]=0\quad\text{ for any }\al,\be=1,\ldots,d-1.
\end{equation*}
We use now the formula
\begin{equation*}
2\g(\D_{e_\al}e_\be,e_\ga)=\g([e_\al,e_\be],e_\ga)+\g([e_\ga,e_\al],e_\be)+\g([e_\ga,e_\be],e_\al)
\end{equation*}
to compute the covariant derivatives
\begin{equation}\label{table1}
\D_{e_\al}e_\be=\delta_{\al\be}e_d,\,\D_{e_\al}e_d=-e_\al,\,\D_{e_d}e_\al=\D_{e_d}e_d=0,\quad\text{ for }\al,\be=1,\ldots,d-1.
\end{equation}

Let $r:\H^d\to[0,\infty)$, $r(x)=d(x,{\bf{0}})$ denote the Riemannian distance function to the origin. We have
\begin{equation}\label{posdef}
\D^2 r\geq 0.
\end{equation}
This is well known consequence of the fact that $\H^d$ has negative constant sectional curvature; it can also be established by explicit computations using \eqref{table1} and the identity $\ch r=\ch s+e^{-s}|v|^2/2$.

The group ${\mathbb G}=SO(d,1)$ is semisimple (thus unimodular) and the group ${\mathbb K}=SO(d)$ is compact. We normalize the Haar measures on ${\mathbb K}$ and ${\mathbb G}$ such that $\int_{\mathbb K} 1\,dk=1$ and 
\begin{equation*}
\int_{\mathbb G}f(g\cdot {\bf{0}})\,dg=\int_{\H^d}f(x)\,d\mu
\end{equation*}
for any $f\in C_0(\H^d)$. Given two functions $f_1,f_2\in C_0({\mathbb G})$ we define the convolution
\begin{equation}\label{convo}
(f_1\ast f_2)(h)=\int_{\mathbb G}f_1(g)f_2(g^{-1}h)\,dg.
\end{equation}
A function $f:{\mathbb G}\to\mathbb{C}$ is called ${\mathbb K}$-biinvariant if 
\begin{equation}\label{inv1}
f(k_1gk_2)=f(g)\text{ for any }k_1,k_2\in\mathbb{K}.
\end{equation}
Similarly, a function $f:\H^d\to\mathbb{C}$ is called ${\mathbb K}$-invariant (or radial) if 
\begin{equation}\label{inv2}
f(k\cdot x)=f(x)\text{ for any }k\in\mathbb{K}\text{ and }x\in\H^d.
\end{equation}
If $f,K\in C_0(\H^d)$ and $K$ is $\mathbb{K}$-invariant then we define (compare to \eqref{convo})
\begin{equation}\label{convo2}
(f\ast K)(x)=\int_{\mathbb G}f(g\cdot{\bf{0}})K(g^{-1}\cdot x)\,dg.
\end{equation}

We record one more integral formula on $\H^d$, which corresponds to the decomposition $\mathbb{G}=\mathbb{K}\mathbb{A}_+\mathbb{K}$: if $f\in C_0(\H^d)$ then
\begin{equation}\label{car}
\int_{\H^d}f(x)\,d\mu=C\int_{\mathbb{K}}\int_{\R_+}f(ka_s\cdot{\bf{0}})(\sh s)^{d-1}\,dkdt,
\end{equation}
where $C$ is a suitable constant and $a_s$ is defined in \eqref{subgroupa}. Thus, if $f$ is $\mathbb{K}$-invariant then
\begin{equation}\label{car'}
\int_{\H^d}f(x)\,d\mu=C\int_{\R_+}f(a_s\cdot{\bf{0}})(\sh s)^{d-1}\,ds.
\end{equation}

\subsection{The Fourier transform on hyperbolic spaces}
The Fourier transform (as defined by Helgason \cite{He3}) takes suitable functions defined on $\H^d$ to functions defined on $\R\times\S^{d-1}$. For $\omega\in \S^{d-1}$ (in the general settting of noncompact symmetric spaces, $\omega\in\mathbb{K}/ \mathbb{M}$ where $\mathbb{M}$ is the centralizer of $\mathbb{A}$ in $\mathbb{K}$) and $\lambda\in\C$, let $b(\omega)=(1,\omega)\in\R^{d+1}$ and 
\begin{equation*}
h_{\lambda,\omega}:\H^d\to\C,\quad h_{\lambda,\omega}(x)=[x,b(\omega)]^{i\lambda-\rho},
\end{equation*}
where
\begin{equation*}
\rho=(d-1)/2.
\end{equation*}
It is known that
\begin{equation}\label{fou}
\Delta h_{\lambda,\omega}=-(\lambda^2+\rho^2)h_{\lambda,\omega},
\end{equation}
where $\Delta$ is the Laplace-Beltrami operator on $\H^d$. The Fourier transform of $f\in C_0(\H^d)$ is defined by the formula
\begin{equation}\label{fht}
\widetilde{f}(\lambda,\omega)=\int_{\H^d}f(x)h_{\lambda,\omega}(x)\,d\mu=\int_{\H^d}f(x)[x,b(\omega)]^{i\lambda-\rho}\,d\mu.
\end{equation}
This transformation admits a Fourier inversion formula: if $f\in C^\infty_0(\H^d)$ then
\begin{equation}\label{Finv}
f(x)=\int_0^\infty\int_{\S^{d-1}}\widetilde{f}(\lambda,\omega)[x,b(\omega)]^{-i\lambda-\rho}|\c(\lambda)|^{-2}\,d\lambda d\omega,
\end{equation}
where, for a suitable constant $C$,
\begin{equation*}
\c(\lambda)=C\frac{\Gamma(i\lambda)}{\Gamma(\rho+i\lambda)}
\end{equation*}
is the Harish-Chandra $\c$-function on $\H^d$, and the invariant measure of $\S^{d-1}$ is normalized to $1$. It follows from \eqref{fou} that
\begin{equation}\label{fou2}
\widetilde{\Delta f}(\lambda,\omega)=-(\lambda^2+\rho^2)\widetilde{f}(\lambda,\omega).
\end{equation}
We record also the nontrivial identity
\begin{equation*}
\int_{\S^{d-1}}\widetilde{f}(\lambda,\omega)[x,b(\omega)]^{-i\lambda-\rho}d\omega=\int_{\S^{d-1}}\widetilde{f}(-\lambda,\omega)[x,b(\omega)]^{i\lambda-\rho}d\omega
\end{equation*}
for any $f\in C^\infty_0(\H^d)$, $\lambda\in \C$, and $x\in\H^d$.

According to the Plancherel theorem, the Fourier transform $f\to\widetilde{f}$ extends to an isometry of $L^2(\H^d)$ onto $L^2(\R_+\times\S^{d-1},|\c(\lambda)|^{-2}d\lambda d\omega)$; moreover
\begin{equation}\label{Plancherel}
\int_{\H^d}f_1(x)\overline{f_2(x)}\,d\mu=\frac{1}{2}\int_{\R\times\S^{d-1}}\widetilde{f_1}(\lambda,\omega)\overline{\widetilde{f_2}(\lambda,\omega)}|\c(\lambda)|^{-2}\,d\lambda d\omega,
\end{equation}
for any $f_1,f_2\in L^2(\H^d)$. As a consequence, any bounded multiplier $m:\R_+\to\mathbb{C}$ defines a bounded operator $T_m$ on $L^2(\H^d)$ by the formula
\begin{equation}\label{defop}
\widetilde{T_m(f)}(\lambda,\omega)=m(\lambda)\cdot \widetilde{f}(\lambda,\omega).
\end{equation}
The question of $L^p$ boundedness of operators defined by multipliers as in \eqref{defop} is more delicate if $p\neq 2$. A necessary condition for boundedness on $L^p(\H^d)$ of the operator $T_m$ is that the multiplier $m$ extend to an even analytic function in the interior of the region $\mathcal{T}_p=\{\lambda\in\C:|\Im\lambda|<|2/p-1|\rho\}$ (see \cite{ClSt}). Conversely, if $p\in(1,\infty)$ and $m:\mathcal{T}_p\to\mathbb{C}$ is an even analytic function which satisfies the symbol-type bounds
\begin{equation}\label{difeq}
|\partial^\alpha m(\lambda)|\leq C(1+|\lambda|)^{-\alpha}\text{ for any }\al\in[0,d+2]\cap\mathbb{Z}\text{ and }\lambda\in\mathcal{T}_p,
\end{equation}  
then $T_m$ extends to a bounded operator on $L^p(\H^d)$ (see \cite{StTo}, and also \cite{ClSt}, \cite{An}, and \cite{Io2} for related statements at various levels of generality).

Assume now that $f\in C_0(\H^d)$ is $\mathbb{K}$-invariant, as in \eqref{inv2}. In this case the formula \eqref{fht} becomes
\begin{equation}\label{fhtrad}
\widetilde{f}(\lambda,\omega)=\widetilde{f}(\lambda)=\int_{\H^d}f(x)\Phi_{-\lambda}(x)\,d\mu,
\end{equation}
where 
\begin{equation}\label{elem}
\Phi_\lambda(x)=\int_{\S^{d-1}}[x,b(\omega)]^{-i\lambda-\rho}\,d\omega
\end{equation}
is the elementary spherical function. The Fourier inversion formula \eqref{Finv} becomes
\begin{equation}\label{Finvrad}
f(x)=\int_0^\infty\widetilde{f}(\lambda)\Phi_\lambda(x)|\c(\lambda)|^{-2}\,d\lambda,
\end{equation}
for any $\mathbb{K}$-invariant function $f\in C^\infty_0(\H^d)$. With the convolution defined as in \eqref{convo2}, we have the important identity
\begin{equation}\label{keyiden}
\widetilde{(f\ast K)}(\lambda,\omega)=\widetilde{f}(\lambda,\omega)\cdot\widetilde{K}(\lambda)
\end{equation}
for any $f,K\in C_0(\H^d)$, provided that $K$ is $\mathbb{K}$-invariant\footnote{Unlike in Euclidean Fourier analysis, there is no simple identity of this type without the assumption that $K$ is $\mathbb{K}$-invariant.}. The definition \eqref{elem} shows easily that $\sup_{\lambda\in\mathbb{R}}|\Phi_\lambda(x)|=\Phi_0(x)$. Thus, using the Plancherel theorem and \eqref{keyiden}
\begin{equation}\label{keyiden2}
||f\ast K||_{L^2(\H^d)}\leq ||f||_{L^2(\H^d)}\cdot\int_{\H^d}|K(x)|\Phi_0(x)\,d\mu
\end{equation}
for any $f,K\in C_0(\H^d)$, if $K$ is $\mathbb{K}$-invariant.

Given a $\mathbb{K}$-invariant (i.e. radial) function $f$ on $\H^d$ we define, by abuse of notation, $f(t):=f(ka_t\cdot{\bf{0}})$ for $t\in[0,\infty)$, where $k\in\mathbb{K}$ and $a_t$ is as in \eqref{subgroupa}. With this convention, the formula \eqref{elem} becomes, for $r\geq 0$,
\begin{equation}\label{reduced}
\Phi_\lambda (r)=C\int_{0}^\pi(\ch r-\sh r\cos\theta)^{-i\lambda-\rho}(\sin\theta)^{d-2}\,d\theta.
\end{equation}
Using this identity, it is easy to see that $\Phi_0(r)\leq Ce^{-\rho r}(r+1)$ for any $r\in\R_+$. It follows from \eqref{keyiden2} and \eqref{car} that
\begin{equation}\label{special}
||f\ast K||_{L^2(\H^d)}\leq C||f||_{L^2}\cdot\int_{0}^\infty |K(r)|e^{-\rho r}(r+1)(\sh r)^{2\rho}\,dr, 
\end{equation}
for any $f,K\in C_0(\H^d)$, if $K$ is $\mathbb{K}$-invariant. This inequality, which is a simple form of the more general Kunze--Stein phenomenon, plays a key role in the proof of the Strichartz estimates in section \ref{Strichartz} (in Lemma \ref{coro}).

We define now the inhomogeneous Sobolev spaces on $\H^d$. There are two possible definitions: using the Riemannian structure $\g$ or using the Fourier transform. These two definitions agree (see, for example, \cite[Section 3]{Ta}). In view of \eqref{fou2}, for $s\in\mathbb{C}$ we define the operator $(-\Delta)^{s/2}$ as given by the Fourier multiplier $\lambda\to (\lambda^2+\rho^2)^{s/2}$. For $p\in(1,\infty)$ and $s\in\R$ we define the Sobolev space $W^{p,s}(\H^d)$ as the closure of $C^\infty_0(\H^d)$ under the norm
\begin{equation*}
||f||_{W^{p,s}(\H^d)}=||(-\Delta)^{s/2}f||_{L^p(\H^d)}.
\end{equation*}
For $s\in \R$ let $H^s=W^{2,s}$. This definition is equivalent to the usual definition of the Sobolev spaces on Riemannian manifolds (this is a consequence of the fact that the operator $(-\Delta)^{s/2}$ is  bounded on $L^p(\H^d)$ for any $s\in\C$, $\Re s\leq 0$, since its symbol satisfies the differential inequalities \eqref{difeq}). In particular, for $s=1$ and $p\in(1,\infty)$
\begin{equation}\label{sobiden}
||f||_{W^{p,1}(\H^d)}=||(-\Delta)^{1/2}f||_{L^p(\H^d)}\approx\Big[\int_{\H^d}|\D^\al f\D_\al f|^{p/2}\,d\mu\Big]^{1/p}=\Big[\int_{\H^d}|\nabla f|^{p}\,d\mu\Big]^{1/p}.
\end{equation}
We record also the Sobolev embedding theorem
\begin{equation*}
W^{p,s}\hookrightarrow L^q\text{ if }1<p\leq q<\infty\text{ and }s=d/p-d/q.
\end{equation*}

\section{Strichartz estimates on hyperbolic spaces}\label{Strichartz}

In this section we prove our main Strichartz estimates. For any $\psi\in H^s(\H^d)$, $s\in\R$, let $W(t)\psi\in C(\R:H^s(\H^d))$ denote the solution of the free Schr\"{o}dinger evolution with data $\psi$, i.e.
\begin{equation*}
\widetilde{W(t)\psi}(\lambda,\omega)=\widetilde{\psi}(\lambda,\omega)\cdot e^{-it(\lambda^2+\rho^2)}.
\end{equation*}

\begin{proposition}\label{StricEst} Assume $q\in(2,(2d+4)/d]$, $I=(a,b)\subseteq \R$ is a bounded open interval, $r=2dq/(dq-4)$, and $S^0_q(I)$ is defined as in \eqref{keyspace1}.
 
(i) If $\phi\in L^2(\H^d)$ then
\begin{equation}\label{Strichartz1}
\|W(t)\phi\|_{S^0_q(I)}\leq C_q\|\phi\|_{L^2},
\end{equation}
for some constant $C_{q}$ that does not depend on the interval $I$.

(ii) If $F\in L^{p_1,p_2}_I$ for some $(p_1,p_2)\in\{(1,2),(q',r'),(q',q')\}$ then
\begin{equation}\label{Strichartz2}
\Big|\Big|\int_{a}^tW(t-s)F(s)\,ds\Big|\Big|_{S^0_q(I)}\leq C_q||F||_{L^{p_1,p_2}_I},
\end{equation}
for some constant $C_q$ that does not depend on the interval $I$.
\end{proposition}

These Strichartz estimates have been proved in the case of radial functions in \cite{BaCaSt} (see also \cite{Ba} and \cite{Pi} for weighted Strichartz estimates for radial functions). The main ingredient in the proof of Proposition \ref{StricEst} is the following pointwise bound:

\begin{lemma}\label{KerBound}
For any $\ep>0$, $r\in[0,\infty)$, and $t\in\R$, we have
\begin{equation}\label{Ke1}
\Big|\int_{\R}e^{-(it+\ep^2)\lambda^2}\Phi_\lambda(r)|\c(\lambda)|^{-2}\,d\lambda\Big|\leq C(|t|^{-d/2}+|t|^{-1})e^{-\rho r}(1+r)^{\rho+1}.
\end{equation}
\end{lemma}

\begin{proof}[Proof of Lemma \ref{KerBound}] The bound \eqref{Ke1} can be improved when $|t|\geq 1$ to give $|t|^{-3/2}$ dispersive decay (as observed in \cite{Ba}). For simplicity, we only prove the weaker dispersive decay stated in \eqref{Ke1}, which still suffices for Proposition \ref{StricEst}.

The notation is described in section \ref{preliminaries}: $\rho=(d-1)/2$,
\begin{equation*}
\c(\lambda)=C\frac{\Gamma(i\lambda)}{\Gamma(\rho+i\lambda)},
\end{equation*}
is the Harish-Chandra $\c$-function, and
\begin{equation*}
\Phi_\lambda (r)=C\int_{0}^\pi(\ch r-\sh r\cos\theta)^{-i\lambda-\rho}(\sin\theta)^{2\rho-1}\,d\theta
\end{equation*}
is the elementary spherical function.

It follows easily from the definition (see also \cite[Proposition A1]{Io}) that $|\c(\lambda)|^{-2}=\c(\lambda)^{-1}\c(-\lambda)^{-1}$ on $\R$, and we have the symbol-type bounds
\begin{equation}\label{Ke3}
\Big|\frac{\partial^\alpha}{\partial\lambda^\alpha}(\lambda^{-1}\c(\lambda)^{-1})\Big|\leq C(1+|\lambda|)^{\rho-1-\alpha}\text{ for any }\alpha\in[0,d+2]\cap\mathbb{Z}\text{ and }\lambda\in\R.
\end{equation}
Also, for $r\geq 1/10$, $\Phi_\lambda(r)$ can be written in the form
\begin{equation}\label{Ke4}
\Phi_\lambda(r)=e^{-\rho r}[e^{i\lambda r}\c(\lambda)m_1(\lambda,r)+e^{-i\lambda r}\c(-\lambda)m_1(-\lambda,r)],
\end{equation}
where
\begin{equation}\label{Ke5}
\Big|\frac{\partial^\alpha}{\partial\lambda^\alpha}m_1(\lambda,r)\Big|\leq C(1+|\lambda|)^{-\alpha}\text{ for any }\alpha\in[0,d+2]\cap\mathbb{Z}\text{ and }\lambda\in\R.
\end{equation}
Finally, for $r\leq 1$,  $\Phi_\lambda(r)$ can be written in the form
\begin{equation}\label{Ke6}
\Phi_\lambda(r)=e^{i\lambda r}m_2(\lambda,r)+e^{-i\lambda r}m_2(-\lambda,r),
\end{equation}
where
\begin{equation}\label{Ke7}
\Big|\frac{\partial^\alpha}{\partial\lambda^\alpha}m_2(\lambda,r)\Big|\leq C(1+r|\lambda|)^{-\rho}(1+|\lambda|)^{-\alpha}\text{ for any }\alpha\in[0,d+2]\cap\mathbb{Z}\text{ and }\lambda\in\R.
\end{equation}
The representations \eqref{Ke4} and \eqref{Ke6} and the bounds \eqref{Ke5} and \eqref{Ke7} follow from \cite{StTo} (or \cite[Proposition A2]{Io}).

The bound \eqref{Ke1} follows now from standard estimates for oscillatory integrals. For suitable choices of $A$, $B$, $m$, and $m'$, we will use the following simple bounds: if $A,B\in\mathbb{R}$, $m\in C^1(\R)$ is a compactly supported function, and $m':\R\to\mathbb{C}$ is supported in the $[-2,-1/2]\cup[1/2,2]$ and satisfies the bounds $|\partial^\alpha_\lambda m'(\lambda)|\leq 1$ for all $\alpha\in[0,d+2]\cap\mathbb{Z}$, then
\begin{equation}\label{Ke29}
\Big|\int_\R e^{i(A\lambda^2+B\lambda)}m(\lambda)\,d\lambda\Big|\leq C|A|^{-1/2}\int_{\R}|\partial_\lambda m(\lambda)|\,d\lambda,
\end{equation}
and
\begin{equation}\label{Ke30}
\Big|\int_\R e^{i(A\lambda^2+B\lambda)}m'(\lambda)\,d\lambda\Big|\leq 
\begin{cases}
C(1+|A|)^{-1/2}&\text{ if }|A/B|\in[1/10,10];\\
C(1+|A|+|B|)^{-d}&\text{ if }|A/B|\notin[1/10,10].
\end{cases}
\end{equation}
To prove \eqref{Ke29} we may assume, after a linear change of variables, that $A=\pm 1$ and $B=0$. Then we decompose the integral into two parts, corresponding to $|\lambda|\leq 1$ and $|\lambda|\geq 1$ and integrate by parts in the second integral. To prove \eqref{Ke30} we use \eqref{Ke29} in the case $|A/B|\in[1/10,10]$, and integrate by parts $d$ times in the case $|A/B|\notin[1/10,10]$. 

We fix a smooth even function $\eta_0:\R\to[0,1]$ supported in $[-2,-1/2]\cup[1/2,2]$ with the property that
\begin{equation}\label{Ke15}
\sum_{j\in\mathbb{Z}}\eta_0(\lambda/2^j)=1\text{ for any }\lambda\in\mathbb{R}\setminus\{0\}.
\end{equation}
for any $j\in\mathbb{Z}$ let $\eta_j(\lambda)=\eta_0(\lambda/2^j)$ and $\eta_{\leq j}=\sum_{j'\leq j}\eta_{j'}$.

We prove now the bound \eqref{Ke1}. Assume first that $r\geq 1/2$. Using \eqref{Ke4} it suffices to prove that
\begin{equation}\label{Ke8}
\Big|\int_{\R}e^{-i(t\lambda^2-r\lambda)}e^{-\ep^2\lambda^2}m_1(\lambda,r)\c(-\lambda)^{-1}\,d\lambda\Big|\leq C(|t|^{-d/2}+|t|^{-1})(1+r)^{\rho+1}.
\end{equation}
Let $J$ denote the smallest integer with the property that $2^J\geq 2^{10}|r/t|$. Assume first that $|t|\leq 2r$, thus $J\geq 0$. Using \eqref{Ke29}, \eqref{Ke3}, and \eqref{Ke5},
\begin{equation*}
\Big|\int_{\R}e^{-i(t\lambda^2-r\lambda)}\eta_{\leq J}(\lambda)e^{-\ep^2\lambda^2}m_1(\lambda,r)\c(-\lambda)^{-1}\,d\lambda\Big|\leq C|t|^{-1/2}2^{\rho J}\leq C|t|^{-d/2}r^\rho.
\end{equation*}
Using the second bound in \eqref{Ke30}, \eqref{Ke3}, and \eqref{Ke5}, for any $j\geq J+1$
\begin{equation*}
\Big|\int_{\R}e^{-i(t\lambda^2-r\lambda)}\eta_{j}(\lambda)e^{-\ep^2\lambda^2}m_1(\lambda,r)\c(-\lambda)^{-1}\,d\lambda\Big|\leq C2^{(\rho+1)j}(2^{2j}|t|)^{-d}\leq C|t|^{-d}2^{-j(2d-\rho-1)}. 
\end{equation*}
We sum the last two bounds to prove \eqref{Ke8} in the case $|t|\leq 2r$.

Assume now that $|t|\geq 2r$, thus $t\geq 1$ and $J\leq 10$. Using again \eqref{Ke29}, \eqref{Ke3} (notice that $|\c(-\lambda)^{-1}|\leq C|\lambda|$ for $|\lambda|\leq C$), and \eqref{Ke5},
\begin{equation*}
\Big|\int_{\R}e^{-i(t\lambda^2-r\lambda)}\eta_{\leq J}(\lambda)e^{-\ep^2\lambda^2}m_1(\lambda,r)\c(-\lambda)^{-1}\,d\lambda\Big|\leq C|t|^{-1/2}2^{J}\leq C|t|^{-3/2}r.
\end{equation*}
 Using the second bound in \eqref{Ke30}, \eqref{Ke3}, and \eqref{Ke5}, for any $j$ with $J+1\leq j\leq 0$
\begin{equation*}
\Big|\int_{\R}e^{-i(t\lambda^2-r\lambda)}\eta_{j}(\lambda)e^{-\ep^2\lambda^2}m_1(\lambda,r)\c(-\lambda)^{-1}\,d\lambda\Big|\leq C2^{2j}(1+2^{2j}|t|)^{-d}. 
\end{equation*}
Finally, for $j\geq 0$, using the second bound in \eqref{Ke30}, \eqref{Ke3}, and \eqref{Ke5},
\begin{equation*}
\Big|\int_{\R}e^{-i(t\lambda^2-r\lambda)}\eta_{j}(\lambda)e^{-\ep^2\lambda^2}m_1(\lambda,r)\c(-\lambda)^{-1}\,d\lambda\Big|\leq C2^{(\rho+1)j}(2^{2j}t)^{-d}. 
\end{equation*}
We sum the last three bounds to prove \eqref{Ke8} in the case $|t|\geq 2r$.

Assume now that $r\leq 1$. Using \eqref{Ke5} it suffices to prove that
\begin{equation}\label{Ke10}
\Big|\int_{\R}e^{-i(t\lambda^2-r\lambda)}e^{-\ep^2\lambda^2}m_2(\lambda,r)\c(\lambda)^{-1}\c(-\lambda)^{-1}\,d\lambda\Big|\leq C(|t|^{-d/2}+|t|^{-1}).
\end{equation}
As  before, let $J$ denote the smallest integer with the property that $2^J\geq 2^{10}|r/t|$. Assume first that $|t|\leq r^2$, thus $|t|\leq 1$ and $2^J\geq r^{-1}$. Using \eqref{Ke29}, \eqref{Ke3}, and \eqref{Ke7},
\begin{equation*}
\Big|\int_{\R}e^{-i(t\lambda^2-r\lambda)}\eta_{\leq J}(\lambda)e^{-\ep^2\lambda^2}m_2(\lambda,r)\c(\lambda)^{-1}\c(-\lambda)^{-1}\,d\lambda\Big|\leq C|t|^{-1/2}2^{2\rho J}(r2^J)^{-\rho}\leq C|t|^{-d/2}.
\end{equation*}
Using the second bound in \eqref{Ke30}, \eqref{Ke3}, and \eqref{Ke7}, for $j\geq J+1$
\begin{equation*}
\begin{split}
\Big|\int_{\R}e^{-i(t\lambda^2-r\lambda)}\eta_{j}(\lambda)e^{-\ep^2\lambda^2}m_2(\lambda,r)\c(\lambda)^{-1}\c(-\lambda)^{-1}\,d\lambda\Big|&\leq C2^j2^{2\rho j}(r2^j)^{-\rho}(1+2^{2j}|t|)^{-d}\\
&\leq Cr^{-\rho}|t|^{-d}2^{-j(2d-\rho-1)}.
\end{split}
\end{equation*}
We sum the last two bounds to prove \eqref{Ke10} in the case $|t|\leq r^2$.

Assume now that $r^2\leq |t|\leq 2r$, thus $|t|\leq 1$ and $J\geq 0$. Using \eqref{Ke29}, \eqref{Ke3}, and \eqref{Ke7},
\begin{equation*}
\Big|\int_{\R}e^{-i(t\lambda^2-r\lambda)}\eta_{\leq J}(\lambda)e^{-\ep^2\lambda^2}m_2(\lambda,r)\c(\lambda)^{-1}\c(-\lambda)^{-1}\,d\lambda\Big|\leq C|t|^{-1/2}2^{2\rho J}\leq C|t|^{-d/2}.
\end{equation*}
Using the second bound in \eqref{Ke30}, \eqref{Ke3}, and \eqref{Ke7}, for $j\geq J+1$
\begin{equation*}
\begin{split}
\Big|\int_{\R}e^{-i(t\lambda^2-r\lambda)}\eta_{j}(\lambda)e^{-\ep^2\lambda^2}m_2(\lambda,r)\c(\lambda)^{-1}\c(-\lambda)^{-1}\,d\lambda\Big|&\leq C2^j2^{2\rho j}(1+2^{2j}|t|)^{-d}.
\end{split}
\end{equation*}
We sum the last two bounds to prove \eqref{Ke10} in the case $r^2\leq |t|\leq 2r$.

Finally, assume that $|t|\geq 2r$, thus $J\leq 10$. Using \eqref{Ke29}, \eqref{Ke3} (i.e. $|\c(\lambda)^{-1}\c(-\lambda)^{-1}|\leq C\lambda^2$ if $|\lambda|\leq C$), and \eqref{Ke7},
\begin{equation*}
\Big|\int_{\R}e^{-i(t\lambda^2-r\lambda)}\eta_{\leq J}(\lambda)e^{-\ep^2\lambda^2}m_2(\lambda,r)\c(\lambda)^{-1}\c(-\lambda)^{-1}\,d\lambda\Big|\leq C|t|^{-1/2}2^{2J}\leq C|t|^{-1}.
\end{equation*}
Using the second bound in \eqref{Ke30}, \eqref{Ke3}, and \eqref{Ke7}, for $J+1\leq j\leq 0$
\begin{equation*}
\begin{split}
\Big|\int_{\R}e^{-i(t\lambda^2-r\lambda)}\eta_{j}(\lambda)e^{-\ep^2\lambda^2}m_2(\lambda,r)\c(\lambda)^{-1}\c(-\lambda)^{-1}\,d\lambda\Big|&\leq C2^j2^{2j}(1+2^{2j}|t|)^{-d}.
\end{split}
\end{equation*}
Using the second bound in \eqref{Ke30}, \eqref{Ke3}, and \eqref{Ke7}, for $j\geq 0$
\begin{equation*}
\begin{split}
\Big|\int_{\R}e^{-i(t\lambda^2-r\lambda)}\eta_{j}(\lambda)e^{-\ep^2\lambda^2}m_2(\lambda,r)\c(\lambda)^{-1}\c(-\lambda)^{-1}\,d\lambda\Big|&\leq C2^j2^{2\rho j}(1+2^{2j}|t|)^{-d}.
\end{split}
\end{equation*}
We sum the last three bounds to prove \eqref{Ke10} in the case $|t|\geq 2r$. This completes the proof of the lemma.
\end{proof}

We have the following consequence of Lemma \ref{KerBound}.

\begin{lemma}\label{coro}
Assume $q$ is as in Proposition \ref{StricEst}, $t\in\R$, $p_1\in\{q',r'\}$, $p_2\in\{q,r\}$, and $\phi\in L^{p_1}(\H^d)$. Then
\begin{equation}\label{coro1}
||W(t)\phi||_{L^{p_2}(\H^d)}\leq C_qB(t)||\phi||_{L^{p_1}(\H^d)},
\end{equation}
where 
\begin{equation}\label{coro2}
B(t)=
\begin{cases}
|t|^{-2/q}&\text{ if }|t|\leq 1;\\
|t|^{-1}&\text{ if }|t|\geq 1.
\end{cases}
\end{equation}
\end{lemma}

\begin{proof}[Proof of Lemma \ref{coro}] It suffices to prove that bound \eqref{coro1} for the operator $P_\ep W(t)$, uniformly for $\ep>0$, where $P_\ep$ is the smoothing operator defined by the Fourier multiplier $\lambda\to e^{-\ep^2\lambda^2}$. Assume first that $|t|\leq 1$. Since $P_\ep W(t)$ is defined by the bounded Fourier multiplier $\lambda\to e^{-\ep^2\lambda^2}e^{-it(\lambda^2+\rho^2)}$, we have $$||P_\ep W(t)\phi||_{L^2(\H^d)}\leq C||\phi||_{L^{2}(\H^d)}.$$ In view of \eqref{Ke1}, $$||P_\ep W(t)\phi||_{L^\infty(\H^d)}\leq C|t|^{-d/2}||\phi||_{L^{1}(\H^d)}.$$
Using again \eqref{Ke1} and \eqref{car'}
\begin{equation*}
||P_\ep W(t)\phi||_{L^\infty(\H^d)}\leq \|\phi\|_{L^{q'}(\H^d)}\|K_t\|_{L^q(\H^d)}\leq C_q|t|^{-d/2}\|\phi\|_{L^{q'}(\H^d)},
\end{equation*}
where $K_t$ denotes the radial kernel of the operator $P_\ep W(t)$, i.e.
\begin{equation*}
K_t(r)=ce^{-it\rho^2}\int_{\R}e^{-(it+\ep^2)\lambda^2}\Phi_\lambda(r)|\c(\lambda)|^{-2}\,d\lambda.
\end{equation*}
The bound \eqref{coro1} follows by interpolation from the last three bounds in the case $|t|\leq 1$.

Assume now that $|t|\geq 1$. With $K_t$ as above, we define $K_{t,n}(r)=\mathbf{1}_{[n,n+1]}(r)K_t(r)$ for $n=0,1,2,\ldots$. Using \eqref{Ke1},
\begin{equation*}
||\phi\ast K_{t,n}||_{L^\infty(\H^d)}\leq C|t|^{-1}e^{-\rho n}(n+1)^{\rho+1}||\phi||_{L^1(\H^d)}.
\end{equation*}
Using \eqref{special} and \eqref{Ke1}
\begin{equation*}
||\phi\ast K_{t,n}||_{L^2(\H^d)}\leq C|t|^{-1}(n+1)^{\rho+2}||\phi||_{L^2(\H^d)}.
\end{equation*}
Finally, using \eqref{Ke1} and \eqref{car'}
\begin{equation*}
||\phi\ast K_{t,n}||_{L^\infty(\H^d)}\leq ||\phi||_{L^{q'}(\H^d)}\|K_{t,n}\|_{L^{q}(\H^d)}\leq C|t|^{-1}e^{-\rho n(q-2)/q}(n+1)^{\rho+1}||\phi||_{L^{q'}(\H^d)}.
\end{equation*}
By interpolation between the last three bounds
\begin{equation*}
||\phi\ast K_{t,n}||_{L^{p_2}(\H^d)}\leq C|t|^{-1}e^{-\rho n(q-2)/q}(n+1)^{\rho+2}||\phi||_{L^{p_1}(\H^d)},
\end{equation*}
for $p_1\in\{q',r'\}$ and $p_2\in\{q,r\}$, and the bound \eqref{coro1} follows by summing this bound over integers $n\geq 0$.
\end{proof}

\begin{proof}[Proof of Proposition \ref{StricEst}] The $L^{\infty,2}_I$ bound in \eqref{Strichartz1} follows from the uniform boundedness of $W(t)$ on $L^2(\H^d)$. A standard $TT^\ast$ argument shows that the remaining bounds in \eqref{Strichartz1} and \eqref{Strichartz2} follow from the estimate
\begin{equation}\label{Strichartz3}
\Big|\Big|\int_{a}^bc(t,s)W(t-s)F(s)\,ds\Big|\Big|_{L^{q,p_2}_I}\leq C_q||F||_{L^{q',p_1}_I},
\end{equation} 
for any measurable function $c:[a,b]\times[a,b]\to[-1,1]$, and any $p_1\in\{q',r'\}$, $p_2\in\{q,r\}$. Using \eqref{coro1}, the fact that $B(t)\leq C|t|^{-2/q}$, and the Hardy--Littlewood--Sobolev inequality, the left-hand side of \eqref{Strichartz3} is dominated by
\begin{equation*}
\begin{split}
&\Big|\Big|\int_{a}^b||W(t-s)F(s)||_{L^{p_2}(\H^d)}\,ds\Big|\Big|_{L^q(I)}\leq C_q\Big|\Big|\int_{a}^bB(t-s)||F(s)||_{L^{p_1}(\H^d)}\,ds\Big|\Big|_{L^q(I)}\\
&\leq C_q\Big|\Big|\int_{a}^b|t-s|^{-2/q}||F(s)||_{L^{p_1}(\H^d)}\,ds\Big|\Big|_{L^q(I)}\leq C_p||F||_{L^{q',p_1}_I},
\end{split}
\end{equation*}
which gives \eqref{Strichartz3}.
\end{proof} 

\section{A Morawetz inequality}\label{Mora}

Let
\begin{equation}\label{pexp}
p_\sigma=\min(2\sigma+2,(2d+4)/d).
\end{equation}
The main result in this section is the following Morawetz inequality:

\begin{proposition}\label{MoraIneq}
Assume that $T>0$, $u\in S^1_{p_\sigma}(-T,T)$ and
\begin{equation*}
i\partial_t u+\Delta u=u|u|^{2\sigma}\text{ on }(-T,T)\times\H^d.
\end{equation*}
Then
\begin{equation}\label{mor}
\|u\|_{L^{2\sigma+2}((t_1,t_2)\times\H^d)}^{2\sigma+2}\leq C_\sigma\sup_{t\in[t_1,t_2]}\|u(t)\|_{L^2(\H^d)}\|u(t)\|_{H^1(\H^d)},
\end{equation}
for any $t_1,t_2\in(-T,T)$.
\end{proposition}

The main ingredient in the proof of Proposition \ref{MoraIneq} is the following lemma:

\begin{lemma}\label{functiona}
There is a smooth function $a:\H^d\to[0,\infty)$ with the following properties:
\begin{equation}\label{prp1}
\begin{split}
&\Delta a=1\quad\text{ on }\H^d;\\
&|\nabla a|=|\D^\al a\D_\al a|^{1/2}\leq C\quad\text{ on }\H^d;\\
&\D^2 a\geq 0\quad\text{ on }\H^d.
\end{split}
\end{equation}
\end{lemma}

\begin{proof}[Proof of Lemma \ref{functiona}] We construct the radial function $a(x)=\widetilde{a}(r)$, with $\widetilde{a}(0)=\partial_r\widetilde{a}(0)=0$. The condition $\Delta a=1$ becomes, in polar coordinates,
\begin{equation}\label{a0}
\Big(\partial_r^2+(d-1)\frac{\cosh r}{\sinh r}\partial_r\Big)\widetilde{a}=1.
\end{equation}
By solving this ODE we get
\begin{equation}\label{a1}
\partial_r\widetilde{a}(r)=\frac{1}{(\sinh r)^{d-1}}\int_0^r(\sinh s)^{d-1}\,ds,
\end{equation}
and
\begin{equation}\label{a2}
\widetilde{a}(r)=\int_0^r\Big(\frac{1}{(\sinh s)^{d-1}}\int_0^s(\sinh t)^{d-1}\,dt\Big)\,ds.
\end{equation}
It follows easily from \eqref{a1} that $|\partial_r\widetilde{a}(r)|\leq C\min(1,r)$, thus
\begin{equation*}
|\D^\al a\D_\al a|\leq C.
\end{equation*}
The first two claims in \eqref{prp1} follow.

We show now that
\begin{equation}\label{a4}
\partial_r\widetilde{a}(r)\in [0,\infty)\text{ and }\partial^2_r\widetilde{a}(r)\in [0,\infty).
\end{equation}
This would suffice to prove the last claim in \eqref{prp1}, in view of \eqref{posdef}. The first inequality follows easily from \eqref{a1}. For the second inequality, using \eqref{a0}, we have
\begin{equation*}
\partial_r^2\widetilde{a}(r)=1-(d-1)\frac{\cosh r}{(\sinh r)^d}\int_0^r(\sinh s)^{d-1}\,ds.
\end{equation*}
Thus, for \eqref{a4} it suffices to prove that
\begin{equation*}
\int_0^r(\sinh s)^{d-1}\,ds\leq \frac{(\sinh r)^d}{(d-1)\cosh r}.
\end{equation*}
This inequality holds because the derivative of the function in the right-hand side is
\begin{equation*}
\frac{d(\sinh r)^{d-1}(\cosh r)^2-(\sinh r)^{d+1}}{(d-1)(\cosh r)^2}=(\sinh r)^{d-1}\frac{(d-1)(\cosh r)^2+1}{(d-1)(\cosh r)^2}\geq (\sinh r)^{d-1}.
\end{equation*}
\end{proof}

\begin{proof}[Proof of Proposition \ref{MoraIneq}] 
Formally, we define the Morawetz action
\begin{equation*}
M_a(t)=2\Im\int_{\H^d}\D^\a a(x)\cdot\overline{u}(x)\D_\al u(x)\,d\mu(x),
\end{equation*}
where $a$ is as in Lemma \ref{functiona}. A standard (formal) computation gives
\begin{equation}\label{smothmor}
\partial_tM_a(t)=4\Re\int_{\H^d}\D^\be\D^\al a\cdot\D_\be\overline{u}\cdot \D_\al u\,d\mu-\int_{\H^d}\Delta\Delta a\cdot |u|^2d\mu+\frac{2\sigma}{\sigma+1}\int_{\H^d}
\Delta a\cdot |u|^{2\sigma+2}\,d\mu.
\end{equation}
The bound \eqref{mor} follows (formally) by integrating in time and using Lemma \ref{functiona} and H\"older inequality. We justify all these formal manipulations below, using the hypothesis $u\in S^1_{p_\sigma}(-T,T)$.

For $\ep\in(0,1/10]$ let $u_\ep=P_\ep u$, where $P_\ep$ is the smoothing operator defined by the Fourier multiplier $\lambda\to e^{-\ep^2\lambda^2}$. Let $\psi_\ep:\H^d\to [0,1]$, $\psi_\ep(x)=\eta_{\leq 0}(\ep r)$. With $a$ as in Lemma \ref{functiona}, we define the Morawetz action $M_a:\R\to\mathbb{R}$,
\begin{equation*}
M_a(t)=2\Im\int_{\H^d}\psi_\ep(x)\D^\a a(x)\cdot\overline{u_\ep}(x)\D_\al u_\ep(x)\,d\mu(x).
\end{equation*}
Let $f_\ep=P_\ep(u|u|^{2\sigma})$, thus
\begin{equation*}
\partial_t u_\ep=i\Delta u_\ep-if_\ep\text{ and }\partial_t\overline{u_\ep}=-i\Delta\overline{u_\ep}+i\overline{f_\ep}.
\end{equation*}
We compute
\begin{equation*}
\begin{split}
\partial_t M_a(t)&=2\Im\int_{\H^d}\psi_\ep\D^\al a\cdot (\partial_t\overline{u_\ep}\cdot\D_\al u_\ep+\overline{u_\ep}\cdot \D_\al\partial_t u_\ep)\,d\mu\\
&=2\Im\int_{\H^d}\psi_\ep\D^\al a\cdot [(-i\Delta\overline{u_\ep}+i\overline{f_\ep})\cdot\D_\al u_\ep+\overline{u_\ep}\cdot \D_\al(i\Delta u_\ep-if_\ep)]\,d\mu\\
&=2\Re\int_{\H^d}\psi_\ep\D^\al a\cdot [(-\Delta\overline{u_\ep}+\overline{f_\ep})\cdot\D_\al u_\ep+\overline{u_\ep}\cdot \D_\al(\Delta u_\ep-f_\ep)]\,d\mu\\
&=\int_{\H^d}\psi_\ep \D^\al a\cdot (\overline{u_\ep}\cdot \D_\al\Delta u_\ep+u_\ep\cdot \D_\al\Delta\overline{u_\ep}-\Delta\overline{u_\ep}\cdot \D_\al u_\ep-\Delta u_\ep\cdot\D_\al\overline{u_\ep})\,d\mu\\
&+\int_{\H^d}\psi_\ep\D^\al a\cdot (\overline{f_\ep}\cdot\D_\al u_\ep+f_\ep\cdot\D_\al\overline{u_\ep}-\overline{u_\ep}\cdot\D_\al f_\ep-u_\ep\cdot\D_\al\overline{f_\ep})\,d\mu\\
&=I+II.
\end{split}
\end{equation*}
By integration by parts and using $\D_\al\D_\be v=\D_\be\D_\al v$ for any scalar $v$, we compute
\begin{equation*}
\begin{split}
&I=\int_{\H^d}-[\D_\al(\psi_\ep\D^\al a)](\overline{u_\ep}\Delta u_\ep+u_\ep\Delta\overline{u_\ep})-2\psi_\ep\D^\al a(\Delta\overline{u_\ep}\cdot \D_\al u_\ep+\Delta u_\ep\cdot\D_\al\overline{u_\ep})\,d\mu\\
&=\int_{\H^d}-(\psi_\ep\Delta a+\D_\al\psi_\ep\D^\al a)[\Delta(u_\ep\overline{u_\ep})-2\D_\be u_\ep\D^\be \overline{u_\ep}]\,d\mu\\
&-2\int_{\H^d}\psi_\ep\D^\al a(\D^\be\D_\be\overline{u_\ep}\cdot \D_\al u_\ep+\D^\be\D_\be u_\ep\cdot\D_\al\overline{u_\ep})\,d\mu\\
&=2\int_{\H^d}\D^\be(\psi_\ep\D^\al a)\cdot (\D_\be\overline{u_\ep}\D_\al u_\ep+\D_\be u_\ep\D_\al\overline{u_\ep})+\int_{\H^d}-\Delta(\psi_\ep\Delta a+\D_\al\psi_\ep\D^\al a)\cdot(u_\ep\overline{u_\ep})d\mu\\
&+\int_{\H^d}2(\psi_\ep\Delta a+\D_\al\psi_\ep\D^\al a)\cdot \D_\be u_\ep\D^\be \overline{u_\ep}+2\psi_\ep\D^\al a(\D_\be\overline{u_\ep}\cdot \D_\al \D^\be u_\ep+\D_\be u_\ep\cdot\D_\al\D^\be\overline{u_\ep})\,d\mu\\
&=2\int_{\H^d}(\psi_\ep\D^\be\D^\al a+\D^\be\psi_\ep\D^\al a)(\D_\be\overline{u_\ep}\cdot \D_\al u_\ep+\D_\be u_\ep\cdot\D_\al\overline{u_\ep})\,d\mu\\
&+\int_{\H^d}-\Delta(\psi_\ep\Delta a+\D_\al\psi_\ep\D^\al a)\cdot(u_\ep\overline{u_\ep})d\mu=A+B,
\end{split}
\end{equation*}
since $\D_\be\overline{u_\ep}\cdot \D_\al \D^\be u_\ep+\D_\be u_\ep\cdot\D_\al\D^\be\overline{u_\ep}=\D_\al(\D_\be u_\ep\D^\be\overline{u_\ep})$. We write $f_\ep=u_\ep|u_\ep|^{2\sigma}+g_\ep$ and use the identity $\overline{u_\ep}|u_\ep|^{2\sigma}\cdot\D_\al u_\ep+u_\ep|u_\ep|^{2\sigma}\cdot\D_\al\overline{u_\ep}=(1/(\sigma+1))\D_\al( |u_\ep|^{2\sigma+2})$ to compute
\begin{equation*}
\begin{split}
II&=2\int_{\H^d}\psi_\ep\D^\al a(\overline{f_\ep}\cdot\D_\al u_\ep+f_\ep\cdot\D_\al\overline{u_\ep})+\D_\al(\psi_\ep\D^\al a)\cdot (\overline{f_\ep}u_\ep+f_\ep\overline{u_\ep})\,d\mu\\
&=2\int_{\H^d}\psi_\ep\D^\al a(\overline{g_\ep}\cdot\D_\al u_\ep+g_\ep\cdot\D_\al\overline{u_\ep})+\D_\al(\psi_\ep\D^\al a)\cdot (\overline{g_\ep}u_\ep+g_\ep\overline{u_\ep})\,d\mu\\
&+\frac{2\sigma}{\sigma+1}\int_{\H^d}\D_\al(\psi_\ep\D^\al a)\cdot |u_\ep|^{2\sigma+2}\,d\mu=C+D.
\end{split}
\end{equation*}
We integrate these identities on the interval $[t_1,t_2]$ to conclude that
\begin{equation*}
M_a(t_2)-M_a(t_1)=\int_{t_1}^{t_2}(A+B+C+D)\,dt.
\end{equation*}
We use now that $u\in S^1_{p_\sigma}(-T,T)$, thus $\lim_{\ep\to 0}||u_\ep-u||_{S^1_{p_\sigma}(-T',T')}=0$ and, using \eqref{Lpbound}, $$\lim_{\ep\to 0}||g_\ep||_{L^{p'_\sigma}((-T',T')\times\H^d)}=0$$ for any $T'<T$. We let $\ep\to 0$, using \eqref{prp1}, to conclude that
\begin{equation*}
\lim_{\ep\to 0}\int_{t_1}^{t_2}A\,dt=2\int_{\H^d\times[t_1,t_2]}\D^\be\D^\al a\cdot(\D_\be\overline{u}\cdot \D_\al u+\D_\be u\cdot\D_\al\overline{u})\,d\mu dt,
\end{equation*}
\begin{equation*}
\lim_{\ep\to 0}\int_{t_1}^{t_2}B\,dt=\lim_{\ep\to 0}\int_{t_1}^{t_2}C\,dt=0,
\end{equation*}
and
\begin{equation*}
\lim_{\ep\to 0}\int_{t_1}^{t_2}D\,dt=\frac{2\sigma}{\sigma+1}\int_{\H^d\times[t_1,t_2]}|u|^{2\sigma+2}\,d\mu dt.
\end{equation*}
Since $|M_a(t)|\leq C\sup_{t\in[t_1,t_2]}\|u(t)\|_{L^2(\H^d)}\|u(t)\|_{H^1(\H^d)}$ (using \eqref{prp1}), it follows that
\begin{equation*}
\begin{split}
2\int_{\H^d\times[t_1,t_2]}\D^\be\D^\al a\cdot(\D_\be\overline{u}\cdot \D_\al u+\D_\be u\cdot\D_\al\overline{u})\,d\mu dt+\frac{2\sigma}{\sigma+1}\int_{\H^d\times[t_1,t_2]}|u|^{2\sigma+2}\,d\mu dt\\
\leq C\sup_{t\in[t_1,t_2]}\|u(t)\|_{L^2(\H^d)}\|u(t)\|_{H^1(\H^d)}
\end{split}
\end{equation*}
The proposition follows using the last inequality in \eqref{prp1}.
\end{proof}

\section{Proof of the main theorem}\label{proofScat}

In this section we complete the proof of Theorem \ref{Main1}. Part (a) follows from the standard Kato method \cite{Ka}, see also Chapter 4 in  \cite{Cazenave:book}, and the Euclidean-type Strichartz estimates 
\begin{equation*}
\|u\|_{L^{\infty,2}_{(-T,T)}\cap L^{q,r}_{(-T,T)}}\leq C_q\|u(0)\|_{L^2(\H^d)}+C_q\|f\|_{{L^{1,2}_{(-T,T)}}+L^{q',r'}_{(-T,T)}},
\end{equation*}
where $f=(i\partial_t+\Delta)u$ on $(-T,T)\times\H^d$. These Strichartz estimates have been proved in section \ref{Strichartz}.

To prove part (b) of Theorem \ref{Main1}, using the standard argument presented in Section 3.6 in  \cite{Tao:book}, it remains to prove the uniform bound \eqref{sol2}
\begin{equation}\label{sol2.1}
||u||_{S^1_q(-T,T)}\leq C(\sigma,q,||\phi||_{H^1(\H^d)})
\end{equation}
for any solution solution $u\in S^1_q(-T,T)$ of \eqref{eq1.1}. We may assume $q\leq p_\sigma$, where $p_\sigma$ is defined in \eqref{pexp}. The main ingredient is the uniform bound in Proposition \ref{MoraIneq}.
Let $\epsilon>0$  be small to be determined later. From the Morawetz estimate in Proposition \ref{MoraIneq} we can divide
the interval $(-T,T)$ into finitely many intervals $I_1,...., I_m$ such that  
\begin{equation}\label{morloc}
\|u\|_{L^{2\sigma+2}(I_j\times\H^d)}\leq \epsilon,
\end{equation}
for all $j=1,...,m$. Fix one of these intervals, and assume it is $I=[a,b)$. Recall the uniform estimate \eqref{Lpbound}
\begin{equation}\label{Lpbound2}
\|(-\Delta)^{1/2}f\|_{L^{p_1}(I\times\H^d)}+\|f\|_{L^{p_2}(I\times\H^d)}\leq C_{p_2}\|f\|_{S^1_q(I)},
\end{equation}    
for any $f\in S^1_q(I)$, $p_1\in[q,(2d+4)/d]$, and $p_2\in[q,(2d+4)/(d-2))$. Using the Duhamel representation of the solution $u$, Proposition \ref{StricEst}, and the fact that $L^{p'_\sigma,p'_\sigma}_{I}\hookrightarrow L^{1,2}_{I}+L^{q',r'}_{I}+L^{q',q'}_I$, we have
\begin{eqnarray}
\|u\|_{S^1_q(I)}&\leq &C_q\|u(a)\|_{H^1}+ C_q\|\,|u|^{2\sigma}|\nabla(u)|\,\|_{L^{p'_\sigma}(I\times\H^d)}\nonumber\\
 \label{scat}
&\leq &C_q\|u(a)\|_{H^1}+C_q\|(-\Delta)^{1/2}u\|_{L^{p_\sigma}(I\times\H^d)}\|u\|_{L^{2\sigma q_\sigma}(I\times\H^d)}^{2\sigma},
\end{eqnarray}
where $$q_\sigma=\frac{p_\sigma p'_\sigma}{p_\sigma-p'_\sigma}=\frac{\max(2\sigma+2,(d+2)\sigma)}{2\sigma}.$$
Since $\sigma\in(0,2/(d-2))$, we observe that
$$2\sigma q_\sigma\in[2\sigma+2,(2d+4)/(d-2)),$$
thus, by interpolation and using \eqref{morloc}, we can continue \eqref{scat} with
\begin{equation}\label{scat2}
\begin{split}
&\leq C(q,\|\phi\|_{H^1})+C_{q,\sigma}\|u\|_{S^1_q(I)}^{1+2\sigma(1-\theta)}\|u\|_{L^{2\sigma+2}(I\times\H^d)}^{\theta\cdot2\sigma}\\
&\leq C(q,\|\phi\|_{H^1})+C_{q,\sigma}\epsilon^{\theta\cdot2\sigma}\|u\|_{S^1_q(I)}^{1+2\sigma(1-\theta)}.
\end{split}
\end{equation}
A continuity argument applied for $\epsilon=\epsilon(\sigma,q,\|\phi\|_{H^1})$ small enough gives the uniform bound \eqref{sol2.1} when $S^1_q(-T,T)$ is replaced by $S^1_q(I)$. Repeting this for the $m$ intervals gives the total bound \eqref{sol2.1}.

\end{document}